\documentclass[11pt]{amsart}

\usepackage{amsmath,amsthm,amsfonts,latexsym,amscd,amssymb,enumerate,color,hyperref}
\usepackage{pdftricks}
\input{xypic}

\usepackage{rotating}
\usepackage{multirow}
\usepackage{graphicx}
\begin{psinputs}
 \usepackage[usenames,dvipsnames]{pstricks}
 \usepackage{epsfig}
  \usepackage{pst-grad} 
 \usepackage{pst-plot} 
\end{psinputs}
\usepackage{hyperref}
\usepackage{epstopdf}

\usepackage{booktabs}
\usepackage{enumitem}
\usepackage{multicol}

\def \<{\langle}
\def \>{\rangle}

\newcommand{\Z}{\mathbb{Z}}
\newcommand{\R}{\mathbb{R}}

\newcommand{\SU}{\mathrm{SU}}
\newcommand{\Sp}{\mathrm{Sp}}
\newcommand{\SO}{\mathrm{SO}}
\newcommand{\Spin}{\mathrm{Spin}}
\newcommand{\I}{\mathrm{I}}
\newcommand{\G}{\mathrm{G}}
\newcommand{\HH}{\mathrm{H}}
\newcommand{\K}{\mathrm{K}}
\newcommand{\RS}{\mathrm{S}}
\newcommand{\T}{\mathrm{T}}
\newcommand{\LL}{\mathrm{L}}
\newcommand{\N}{\mathrm{N}}
\newcommand{\M}{\mathrm{M}}
\newcommand{\X}{\mathrm{X}}
\newcommand{\Y}{\mathrm{Y}}

\newcommand{\SP}{\mathbf{S}} 
\newcommand{\PS}{\mathbf{P}} 

\newcommand{\Hom}{\mathrm{Hom}}

\newcommand{\mytableextraspace}{\addlinespace[.4em]}

\DeclareMathOperator{\Proj}{Proj}
\DeclareMathOperator{\Isom}{Isom}
\DeclareMathOperator{\Sign}{Sign}
\DeclareMathOperator{\Susp}{Susp}
\DeclareMathOperator{\rank}{rank}

\newtheorem{thm}{Theorem}[section]

\newtheorem{lem}[thm]{Lemma}
\newtheorem{prop}[thm]{Proposition}
\newtheorem{claim}[thm]{Claim}

\newtheorem{theorem}{Theorem}

\newtheorem{corollary}[theorem]{Corollary}

\theoremstyle{definition}
\newtheorem{defn}[thm]{Definition}

\newtheorem{example}[thm]{Example}

\newtheorem*{ack}{Acknowledgements}

\numberwithin{equation}{section}

\begin{document}


\author[F.~Galaz-Garc\'ia]{Fernando Galaz-Garc\'ia}
\address[F.~Galaz-Garc\'ia]{Institut f\"ur Algebra und Geometrie, Karlsruher Institut f\"ur Technologie (KIT), Karlsruhe, Germany.}
\email{galazgarcia@kit.edu}


\author[M.~Zarei]{Masoumeh Zarei} 
\address[M.~Zarei]{Department of Pure Mathematics, Faculty of Mathematical Sciences, Tarbiat Modares University, Tehran, Iran and Institut f\"ur Algebra und Geometrie, Karlsruher Institut f\"ur Technologie (KIT), Karlsruhe, Germany.}
\email{masoumeh.zarei@modares.ac.ir}

\title{Cohomogeneity one topological manifolds revisited}
\date{\today}




\begin{abstract} 
We prove a structure theorem for closed topological manifolds of cohomogeneity one; this result corrects an oversight in the literature.  We complete the equivariant classification of closed, simply connected cohomogeneity one topological manifolds in dimensions $5$, $6$, and $7$ and obtain topological characterizations of these spaces. In these dimensions, these manifolds are homeomorphic to smooth manifolds.
\end{abstract}

\maketitle



\section{Main Results}

A topological manifold with an (effective) topological action of a compact Lie group is of \emph{cohomogeneity one} if its orbit space is one-dimensional. These manifolds were introduced by Mostert \cite{Mo} in 1957 and their topology and geometry have been extensively studied  in the smooth category (see, for example, \cite{EU,Fr,GM,Ho1,Ho2,Ho3,Ne,Pa} and \cite{Bo,DS,De,GVZ,GWZ,GZ00,GZ02,Se,ST,Zi}). Much less attention has been given to these spaces in the topological category, probably because of the assertion in \cite{Mo} that every topological manifold of cohomogeneity one is equivariantly homeomorphic to a smooth manifold. This statement originates in the 
claim that an integral homology sphere that is also a homogeneous space for a compact Lie group must be  a standard sphere (see \cite[Section 2, Corollary]{Mo}). This, however, is not the case. Indeed, the Poincar\'e homology sphere $\PS^3$ is a homogeneous space for the Lie groups $\SU(2)$ and $\SO(3)$ and it can be written as $\PS^3\approx \SU(2)/\I^*\approx \SO(3)/\I$, where $\I^*$ is the binary icosahedral group and $\I$ is the icosahedral group  (see \cite{KS} or \cite[p.~89]{W}). One can combine this fact with the Double Suspension Theorem of Edwards and Cannon \cite{Ca,Ed} to construct topological manifolds with cohomogeneity one actions that are not equivariantly homeomorphic to smooth actions (see Example~\ref{E:C1P3}). We point out that, by work of Bredon \cite{Br}, the Poincar\'e homology sphere is the only integral homology sphere that is also a homogeneous space, besides the usual spheres. In the present article we fix the gap in \cite{Mo} and explore some of its consequences.


Our first result is a complete structure theorem for closed cohomogeneity one topological manifolds (cf.~\cite[Theorem~4]{Mo}). As is customary, we say that a manifold is \emph{closed} if it is compact and has no boundary. We briefly discuss the non-compact case in Section~\ref{S:Preliminaries}.
 

\begin{theorem}
\label{T:STRUCTURE}
Let $\M$ be a closed topological manifold with an (almost) effective topological $\G$ action of cohomogeneity one with principal isotropy $\HH$.  Then the orbit space is homeomorphic to either a closed interval or to a circle, and the following hold.
 \begin{itemize}
 	\item[(1)]  If the orbit space of the action is an interval, then $\M$ is the union of two fiber bundles over the two singular orbits whose fibers are cones over spheres or the Poincar\'e homology sphere, that is,
	\[ 
		\M=\G\times_{\K^{-}}C(\K^{-}/\HH) \cup_{\G/\HH}\G\times_{\K^{+}} C(\K^{+}/\HH).
	\]
	The group diagram of the action is given by $(\G, \HH, \K^{-}, \K^{+})$, where $\K^{\pm}/\HH$ are spheres or the Poincar\'e homology sphere. Conversely, a group diagram $(\G, \HH, \K^{-},\K^{+})$, where $\K^{\pm}/\HH$ are homeomorphic to a sphere, or to the Poincar\'e homology sphere with $\dim \G/\HH\geq 4$, determines a cohomogeneity one topological manifold.
	\\
  	\item[(2)] If the orbit space of the action is a circle, then $\M$ is equivariantly homeomorphic to a $\G/\HH$-bundle over a circle with structure group $N(\HH)/\HH$. 
	\end{itemize}
\end{theorem}

This theorem stands in contrast with the corresponding statement in the smooth category, where the fibers of the bundle decomposition are cones over spheres, i.e.~disks, and the manifold decomposes as a union of two disk bundles. We prove Theorem~\ref{T:STRUCTURE} in Section~\ref{S:Preliminaries}. 
\\

It is well known that a  closed smooth cohomogeneity one $\G$-manifold  admits a $\G$-invariant Riemannian metric with a lower sectional curvature bound. Alexandrov spaces are synthetic generalizations of Riemannian manifolds with curvature bounded below (see \cite{BBI,BGP}). Theorem~\ref{T:STRUCTURE}, in combination with work of Galaz-Garc\'ia and Searle \cite{GS}, implies:

\begin{corollary}
\label{C:EQUIV_ALEX_METRIC}
A closed topological manifold of cohomogeneity one admits an invariant Alexandrov metric.
\end{corollary}

We prove this corollary in Section~\ref{S:EQ_ALEX_METRICS}. It is an open question whether every topological manifold admits an Alexandrov metric. Corollary~\ref{C:EQUIV_ALEX_METRIC} shows that this is true if $\M$ admits a cohomogeneity one action. 
\\

A  topological manifold $\M$ is \emph{smoothable} if it  is homeomorphic to a smooth manifold. A topological $\G$-action on a smoothable topological manifold $\M$ is \emph{smoothable} if it is equivalent to a smooth $\G$-action on $\M$. The following results are simple consequences of Theorem~\ref{T:STRUCTURE}.


\begin{corollary}
\label{T:SMOOTH_CONDITION} A closed cohomogeneity one topological manifold is equivariantly homeomorphic to a smooth manifold  if and only if every slice is homeomorphic to a disk. 
\end{corollary}


\begin{corollary}
\label{C:SLICES}
Let $\M$ be a closed topological manifold with a cohomogeneity one $\G$-action. 
	\begin{enumerate}
		\item If the codimension of the singular orbits is not $4$, then the action is smoothable.
		\item If the codimension of some singular orbit is $4$, then the action is
		\begin{enumerate}
			\item smoothable if and only if the $4$-dimensional slices are homeomorphic to $4$-dimensional disks.
			\item non-smoothable if and only if some $4$-dimensional slice is homeomorphic to the cone over the Poincar\'e homology sphere.
		\end{enumerate}
	\end{enumerate}
\end{corollary}


\begin{corollary}
\label{C:LDim}
Every cohomogeneity one action on a topological $n$-manifold, $n\leq 4$, is smoothable.
\end{corollary}


Closed, smooth manifolds of cohomogeneity one have been classified equivariantly by Mostert \cite{Mo} and Neumann \cite{Ne} in dimensions $2$ and $3$, by Parker \cite{Pa} in dimension $4$, and, assuming simply connectedness, by Hoelscher \cite{Ho2} in dimensions $5$, $6$ and $7$. Mostert, Neumann and Parker gave canonical representatives for the classes in the equivariant classification  in dimension $n\leq 4$. This was also done by Hoelscher \cite{Ho1} in dimensions $5$ and $6$ in the simply connected case. In dimension $7$,  Hoelscher \cite{Ho3}, Escher and Ultman \cite{EU} computed the homology groups of the manifolds appearing in the equivariant classification. By Corollary~\ref{C:LDim}, the classification of closed, cohomogeneity one topological manifolds is complete in dimension $n\leq 4$. In dimensions $5$, $6$ and $7$, however, it follows from Corollary~\ref{C:SLICES} that Hoelscher's results do not yield a classification in the topological category. Our second theorem completes the equivariant classification of closed, cohomogeneity one topological manifolds in these dimensions. 


\begin{theorem}
\label{T:CLASSIF}
Let $\M$ be a closed, simply connected topological $n$-manifold, $n\leq 7$, with an (almost) effective cohomogeneity one action of a compact connected Lie group $G$. If the action is non-smoothable, then it is given by one of the diagrams in Table~\ref{TB:GROUP_DGMS} and $\M$ can be exhibited as one of the manifolds in this table. 
\end{theorem}


\begin{table}[!htbp]
\begin{center}
\small{
\begin{tabular}{p{1.7cm}p{6.9cm}l}\toprule
Dimension 	& Diagram  & Manifold \\
\midrule
$5$	& $(\RS^3\times \RS^1,  \I^*\times \mathbb{Z}_k,  \I^*\times \RS^1, \RS^3\times \mathbb{Z}_k)$	&  $\PS^3\ast \SP^1\approx \SP^5$ \\ \mytableextraspace
\midrule
$6$	& $(\RS^3\times \RS^3, \I^*\times \RS^1,  \RS^3\times \RS^1, \RS^3\times \RS^1)$ 	& $\Susp\PS^3 \times \SP^2$    \\ 
	& 																				& $\approx\SP^4 \times \SP^2$    \\ \mytableextraspace

	 &$(\RS^3\times \RS^3,   \RS^1\times \I^*,   \RS^3\times \I^*, \RS^1\times \RS^3)$ &$\PS^3\ast \SP^2\approx \SP^6$  \\ \mytableextraspace
\midrule
7 	&  $(\RS^3\times \RS^3, \I^*\times 1,  \I^*\times \RS^3, \RS^3\times 1)$ 							& $\PS^3\ast \SP^3\approx \SP^7$ \\ \mytableextraspace
	&$(\RS^3\times \RS^3,   \Delta \I^*,   \I^*\times \RS^3, \Delta \RS^3)$ 							& $\PS^3\ast \SP^3\approx \SP^7$\\  \mytableextraspace
	&$(\RS^3\times \RS^3, \I^*\times  \mathbb{Z}_k,   \I^*\times \RS^1, \RS^3\times \mathbb{Z}_k)$ 	& $\SP^5$ bundle over $\SP^2$\\  \mytableextraspace
	& $(\RS^3\times \RS^3,    \I^*\times \I^*,   \RS^3\times \I^*,  \I^*\times \RS^3)$ 					&$\PS^3\ast \PS^3\approx \SP^7$\\ \mytableextraspace
	&$(\RS^3\times \RS^3,  \I^*\times 1, \RS^3\times 1, \RS^3\times 1)$ 											& $\Susp\PS^3 \times \SP^3$\\ \mytableextraspace
		& 																						& $\approx\SP^4 \times \SP^3$    \\ \mytableextraspace
	&$(\RS^3\times \RS^3, \Delta \I^*,  \Delta \RS^3, \Delta\RS^3)$ 								& $ \Susp  \PS^3 \times \SP^3$\\ \mytableextraspace
	& 																							& $\approx\SP^4 \times \SP^3$    \\ \mytableextraspace
	&$(\RS^3\times \RS^3\times \RS^1, \I^*\times \RS^1\times \mathbb{Z}_k,   \I^*\times \T^2, \RS^3\times \RS^1\times \mathbb{Z}_k)$&  $\SP^5$ bundle over $\SP^2$\\
\bottomrule
\end{tabular}
}
\end{center}
\caption{\label{TB:GROUP_DGMS}Non-smoothable cohomogeneity one actions in dimensions $5$, $6$ and $7$}
\end{table}

The proof of Theorem~\ref{T:CLASSIF} follows the outline of the proofs of the classification in the smooth case (see, for example, \cite{Ho2}). After determining the admissible group diagrams, we can write the manifolds as joins, products or bundles in terms of familiar spaces. The problem still remains whether the topological manifolds in Table~\ref{TB:GROUP_DGMS} are smoothable. One can quickly settle this question for the joins in Table~\ref{TB:GROUP_DGMS} since, by the Double Suspension Theorem, these manifolds are homeomorphic to spheres and therefore they are smoothable. The situation for the products $\Susp \PS^3\times \SP^2$ and $\Susp\PS^3\times \SP^3$ in Table~\ref{TB:GROUP_DGMS}, where $\Susp \PS^3$ denotes the suspension of $\PS^3$, is more delicate. For example, the $6$-dimensional product $\Susp \SP^3\times \SP^2$ is homotopy equivalent to $\SP^4\times \SP^2$.   By the classification of closed, oriented simply connected $6$-dimensional topological manifolds with torsion free homology, carried out by Wall \cite{Wa}, Jupp \cite{J} and Zhubr \cite{Zh}, there exist infinitely many homeomorphism types for a homotopy $\SP^4\times \SP^2$, parametrized by a nonnegative integer $k$. For $k$ even, the corresponding homeomorphism type is smoothable; for $k$ odd, the corresponding homeomorphism type is non-smoothable (see Section~\ref{S:TOP_MFT}). Our third theorem 
settles the smoothability of $\Susp \PS^3\times \SP^2$.



\begin{theorem} The manifold $\Susp \PS^3\times \SP^2$ is homeomorphic to $\SP^4\times\SP^2$.
\label{T:TOP_MFT} 
\end{theorem} 

The proof of Theorem~\ref{T:TOP_MFT} is an application of the classification of closed, oriented simply connected topological $6$-manifolds with torsion free homology, and essentially reduces to computing the first Pontryagin class of \linebreak $\Susp \PS^3\times \SP^2$ (see Section~\ref{S:TOP_MFT}). To do this, we use results  of Zagier \cite{Zag}, Atiyah and Singer's $\G$-signature theorem \cite{AS} and a signature formula of Atiyah and Bott \cite{AB}. 

Observe that $\Susp\PS^3\times\SP^3$ is the total space of a principal $\RS^1$-bundle over $\Susp\PS^3\times\SP^2$. By Theorem~\ref{T:TOP_MFT}, $\Susp\PS^3\times\SP^2\approx \SP^4\times\SP^2$. Hence, $\Susp\PS^3\times\SP^3$ is smoothable and, since the Euler class of the bundle is a generator of $H^2(\SP^4\times\SP^2)$, we obtain the following result. 

\begin{corollary} The manifold $\Susp \PS^3\times \SP^3$ is homeomorphic to $\SP^4\times\SP^3$.
\end{corollary}

Let $\M$ be the total space of a topological $\SP^5$-bundle over $\SP^2$. Since $H^4(\M,\Z_2)=0$, the Kirby-Siebenmann class of $\M$ vanishes, so $\M$ admits a PL structure (see \cite{KSi}). Since, in dimensions $n\leq 7$, every PL $n$-manifold admits at least one compatible smooth structure (see \cite{HM,KM,M} or \cite[p.~66]{We}), it follows that $\M$ is smoothable. Thus, the homeomorphisms in the third column of Table~\ref{TB:GROUP_DGMS}, combined with Corollary~\ref{C:LDim}, yield the following result. 


\begin{corollary}
A  closed, simply connected topological $n$-manifold of cohomogeneity one is homeomorphic to a smooth manifold, provided $n\leq 7$.
\end{corollary}


Our paper is organized as follows. In Section~\ref{S:Preliminaries} we discuss Mostert's article \cite{Mo} and prove Theorem~\ref{T:STRUCTURE}. We prove Corollary~\ref{C:EQUIV_ALEX_METRIC} in Section~\ref{S:EQ_ALEX_METRICS}. In Section~\ref{S:TOOLS} we collect some results on cohomogeneity one topological manifolds that we will use in the proof of Theorem~\ref{T:CLASSIF}. Sections \ref{S:CLASSIF} and \ref{S:TOP_MFT} contain, respectively, the proofs of Theorems~\ref{T:CLASSIF} and \ref{T:TOP_MFT}. 


\begin{ack} M.~Zarei thanks the Institut f\"ur Algebra und Geometrie at the Karlsruher Institut f\"ur Technologie (KIT) for its hospitality while the work presented herein was carried out. Both authors would like to thank Anand Dessai, Marco~Radeschi and Alexey V.~Zhubr  for helpful conversations on the proof of Theorem~\ref{T:TOP_MFT}, and Martin Herrmann, Wilderich Tuschmann, and Burkhard Wilking for conversations on the smoothability of manifolds. The authors also thank Martin Kerin and Wolfgang Ziller for suggesting improvements to the exposition. M.~Zarei was partially supported by the Ministry of Science, Research and Technology of Iran.
\end{ack}


\section{Setup and proof of Theorem~\ref{T:STRUCTURE}}
\label{S:Preliminaries}


\subsection{Notation} Let $\M$ be a topological manifold and let $x$ be a point in $\M$.  Given a topological (left) action $\G\times \M\rightarrow \M$ of a Lie group $\G$, we let $\G(x)=\{\,gx \mid g\in \G \,\}$ be the \emph{orbit}  of $x$ under the action of $\G$. The \emph{isotropy group} of $x$ is the subgroup $\G_x=\{\, g\in \G \mid gx=x\,\}$. Observe that $\G(x)\approx \G/\G_x$. We will denote the orbit space of the action by $\M/\G$ and let  $\pi:\M\rightarrow \M/\G$ be the orbit projection map.  The \emph{(ineffective) kernel} of the action is the subgroup $\K=\bigcap_{x\in \M}\G_x$. The action is \emph{effective} if $\K$  is the trivial subgroup $\{e\}$ of $\G$; the action is \emph{almost effective} if $\K$ is finite.

We will say that two $\G$-manifolds are \emph{equivalent} if they are equivariantly homeomorphic. From now on, we will suppose that $\G$ is compact and assume that the reader is familiar with the basic notions of compact transformation groups (see, for example, Bredon \cite{Bredon}). We will assume all manifolds to be connected. 


As for locally smooth actions (see \cite[Ch.~IV, Section 3]{Br}), for a topological action of $\G$ on $\M$ there also exists a maximum orbit type $\G/\HH$, i.e. $\HH$ is conjugate to a subgroup of each isotropy group. One sees this as follows. Let $\M_0$ be the set of points with isotropy group of smallest dimension and least number of components. By work of Montgomery and Yang \cite{MY}, $\M_0$ is an open, dense and connected subset of $\M$. On the other hand, by work of Montgomery and Zippin \cite{MZ},  for every $x\in \M$ there is a neighborhood $V$ such that $\G_y$ is conjugate to a subgroup of $\G_x$ for $y\in V$. It then follows from the connectedness of $\M_0$ that the isotropy groups $\G_y$, $y\in \M_0$ are conjugate to each other. By the density of $\M_0$ and the existence of the neighborhood $V$, each group $\G_y$, for  $y\in \M_0$, is conjugate to a subgroup of every isotropy group. Therefore, the orbit type $\G/\G_y$, for $y\in M_0$, is maximal. We call this orbit type the \emph{principal orbit type} and orbits of this type \emph{principal orbits}.

A \emph{homology sphere} is a closed  topological $n$-manifold $\M^n$ that is an integral homology sphere, i.e.~$H_*(\M^n,\Z)\cong H_*(\SP^n,\Z)$. We will denote the suspension of a topological space $\X$ by $\Susp \X$ and the join of $\X$ with a topological space $\Y$ by $\X\ast \Y$. Recall that $\Susp \X\approx  \X \ast \SP^0$ and, in general, $\Susp^n \X \approx \X\ast \SP^{n-1}$ for $n\geq 1$.

We will denote the Poincar\'e homology sphere by $\PS^3$; it is homeomorphic to the homogeneous spaces $\SU(2)/\I^*$ and $\SO(3)/\I$, where $\I^*$ is the binary icosahedral group and $\I$ is the icosahedral group. We will use some basic concepts of piecewise-linear topology in the proof of Theorem~A. We refer the reader to \cite{RS} for the relevant definitions.  


\subsection{Cohomogeneity one topological manifolds}

In this subsection we collect basic facts on cohomogeneity one topological manifolds, discuss the omission in Mostert's work \cite{Mo} that gave rise to the present article, and prove some preliminary results that we will use in the proof of Theorems~\ref{T:STRUCTURE} and \ref{T:CLASSIF}.


\begin{defn}
Let $\M$ be a topological $n$-manifold with a topological action of a compact connected Lie group  $\G$. The action is of \emph{cohomogeneity one} if the orbit space is one-dimensional or, equivalently, if there exists an orbit of dimension $n-1$. A topological manifold with a topological action of cohomogeneity one is a \emph{cohomogeneity one manifold}.
\end{defn}

By \cite[Theorem~1]{Mo}, the orbit space of a cohomogeneity one manifold is homeomorphic to a circle, an open interval, a half open interval or a closed interval $[-1,+1]$.  We refer to orbits which map to endpoints as \emph{singular}. We call the isotropy groups of points in these orbits \emph{singular} isotropy groups.  When the orbit space is homeomorphic to $[-1,+1]$, we denote a singular isotropy group corresponding to a point in the orbit $\pm1$ by $\K^\pm$. Orbits that are not singular are called \emph{regular orbits}; they all have the same isotropy group $\HH$ and project to interior points of the orbit space. The subgroup $\HH$ is called the \emph{principal} isotropy group

As indicated in the introduction, the oversight in \cite{Mo} stems from the claim that a homology sphere that is a homogeneous space must be a standard sphere. More precisely, in \cite[Sections 2 and 4]{Mo} Mostert shows that $\K/\HH$, where $\K$  is a singular isotropy group,  must be a homology sphere and a homogeneous space (see~\cite[Lemma~2 and proof of Theorem~2]{Mo}) and concludes, erroneously, that $\K/\HH$ must be a standard sphere (see~\cite[Section 2, Corollary]{Mo}). This, as explained in the introduction, is not the case.
 The following result of Bredon \cite{Br} implies that the Poincar\'e homology sphere $\PS^3$ and  standard spheres are the only possibilities for  $\K/\HH$.


\begin{thm}[Bredon]
\label{T:BREDON_P3}
Let $\G$ be a compact Lie group and $\HH$ a closed subgroup of $\G$.
\begin{itemize}
 \item[(1)] If $\G/\HH$ is a homology $k$-sphere, then $\G/\HH$ is homeomorphic to either $\SP^k$ or to the Poincar\'e homology sphere $\PS^3$. 
\item[(2)] If $\G$ acts almost effectively and transitively on $\PS^3$, then $\G$ is isomorphic to $\SU(2)$ or $\SO(3)$, with $\I^*$ or $\I$ as the isotropy group, respectively.
\end{itemize}
\end{thm}

 The following example shows that there are cohomogeneity one topological manifolds with $\K/\HH\approx \PS^3$.


\begin{example}
\label{E:C1P3}
Let  $\RS^3\times \SO(n+1)$,  $n\geq 1$,  act on $\PS^3\ast \SP^n$ as the join action of the standard transitive actions of $\RS^3\cong \SU(2)$ on $\PS^3$ and  $\SO(n+1)$ on $\SP^n$. The orbit space is homeomorphic to $[-1,+1]$ and  $\K^+=\RS^3\times  \SO(n)$, $\K^-=\I^* \times  \SO(n+1)$ and $\HH=\I^* \times \SO(n)$. Thus $\K^+/\HH=\PS^3$. By the Double Suspension Theorem, $\Susp^2 \PS^3\approx \SP^5$ and it follows that $\PS^3\ast \SP^n\approx\Susp^{n+1} \PS^3$ is homeomorphic to $\SP^{n+4}$.
\end{example}

Taking into account Theorem~\ref{T:BREDON_P3}, the comments preceding it, and Example~\ref{E:C1P3}, we conclude that one must amend \cite[Theorem~4]{Mo} (considering the corrections in the Errata to \cite{Mo}) by adding $\PS^3$ as a second possibility for $\K/\HH$ in items (iii) and (iv) in \cite[Theorem~4]{Mo}.
There is a fifth item in \cite[Theorem~4]{Mo}:


\begin{claim}
\label{C:ITEM_V} 
The action of the group $\G$ on a space with structure as in \cite[Theorem~4]{Mo} is equivalent to a cohomogeneity one $\G$-action on the manifold $\M$ and, conversely, a space $\M$ constructed in such a way is a  topological manifold with a cohomogeneity one action of $\G$. 
\end{claim}

This claim is true when all the $\K^\pm/\HH$ are spheres and follows as in \cite{Mo}. In the case where at least one of the $\K^\pm/\HH$ is homeomorphic to $\PS^3$, one must prove Claim~\ref{C:ITEM_V}.  We do this at the end of this section, in the case where $\M$ is closed (i.e.~compact and without boundary). This yields Theorem~\ref{T:STRUCTURE}. The remaining cases, where the orbit space is not compact, can be dealt with in an analogous way, and we leave this task to the interested reader. 

By item (iv) in \cite[Theorem~4]{Mo}, a cohomogeneity one $\G$ action on a closed topological manifold with orbit space an interval determines a group diagram 
\begin{equation*}
\xymatrix{ & \G    & \\
\K^{-} \ar[ru]^{j_{-}} & & \K^{+} \ar[lu]_{j_{+}}\\
& \HH  \ar[lu]^{i_{-}} \ar[ru]_{i_{+}}& }
\end{equation*}
 where $i_{\pm}$ and $j_{\pm}$ are the inclusion maps, $\K^{\pm}$ are the isotropy groups of the singular orbits at the endpoints of the interval, and $\HH$ is the principal isotropy group of the action.  We will denote this diagram by the $4$-tuple $(\G, \HH, \K^{-}, \K^{+})$. The inclusion maps are an important element in the group diagram, as illustrated by the following simple example: $(\T^2, \{e\}, \T^1, \T^1)$ determines both $\SP^3$ and $\SP^2\times \SP^1$, where in the first case the inclusion maps are to the first and second factors, respectively, of $\T^2$, and in the second case, both inclusion maps are the same (cf. \cite{Ne}). Now we prove Theorem~\ref{T:STRUCTURE}.


\subsection{Proof of Theorem~\ref{T:STRUCTURE}}
\label{S:PF_THM_STRUCT}
Let $\M^n$ be a closed topological $n$-manifold with an (almost) effective topological $\G$ action of cohomogeneity one with principal isotropy $\HH$. 
By item (i) in \cite[Theorem~4]{Mo}, the orbit space is homeomorphic to either a closed interval or to a circle.  Part (2) of Theorem~\ref{T:STRUCTURE} follows from item (ii) in \cite[Theorem~4]{Mo}.
 Therefore, we need only prove part (1) of Theorem~\ref{T:STRUCTURE}, where the orbit space $\M/\G$ is homeomorphic to a closed interval $[-1,+1]$. The ``if'' statement in this case corresponds to part (iv) of \cite[Theorem~4]{Mo} (keeping in mind that one must add $\PS^3$ as a possibility for $\K^\pm/\HH)$. Now we prove the ``only if'' statement. 

Let $(\G, \HH, \K^{-}, \K^{+})$ be a group diagram satisfying the hypotheses of part (1) of Theorem~\ref{T:STRUCTURE}. By the work of Mostert, we need only consider the case where at least one of $\K^{\pm}/\HH$ is the Poincar\'e sphere $\PS^3$. In this case, $n\geq 5$.

Suppose, without loss of generality, that  $\K^+/\HH=\PS^3$. Since $n\geq 5$, the singular orbit $\G/\K^+$ is at least one-dimensional.  Observe now that the space 
\begin{align}
\label{E:DBL_BDL}
 \X=\G\times_{\K^{-}}C(\K^{-}/\HH) \cup_{\G/\HH}\G\times_{\K^{+}} C(\K^{+}/\HH) 
\end{align}
is a finite polyhedron. Since the link of every point in the singular orbit $\G/\K^+$ is $\SP^{n-5}*\PS^3$, the following result  (see \cite[p.~742]{Wu}) implies that $X$ is a topological manifold: 

\begin{thm}[Edwards]
\label{T:EDWARDS_POLY}A finite polyhedron $P$ is a closed topological $n$-manifold if and only if the link of every vertex of $P$ is simply connected if $n\geq 3$, and the link of every point of $P$ has the homology of the $(n-1)$-sphere.
\end{thm}  \hfill$\square$


\section{Existence of Invariant Alexandrov metrics}
\label{S:EQ_ALEX_METRICS} In this section we point out that every closed cohomogeneity one topological manifold admits an invariant Alexandrov metric. Let us first recall some basic facts about Alexandrov spaces, all of which can be found in \cite{BBI}. 

A finite dimensional length space $(\X,d)$ has curvature bounded from below by $k$ if every point $x\in X$ has a neighborhood $U$ such that for any collection of four different points $(x_0,x_1,x_2,x_3)$ in $U$, the following condition  holds:
\[
\angle_{x_{1},x_{2}}(k)+\angle_{x_{2},x_{3}}(k)+\angle_{x_3,x_1}(k)\leq 2\pi.
\]
Here, $\angle_{x_{i},x_{j}}(k)$, called the \emph{comparison angle}, is the angle at $x_{0}(k)$ in the geodesic triangle in $\M^2_k$, the simply-connected $2$-manifold with constant curvature $k$, with vertices $(x_{0}(k),x_{i}(k),x_{j}(k))$, which are the isometric images of $(x_{0},x_{i},x_{j})$. An \emph{Alexandrov space} is a complete length space of curvature bounded below by $k$, for some $k\in \R$. The isometry group $\Isom(\X)$ of an Alexandrov space $\X$ is a Lie group (see \cite{FY}) and $\Isom(\X)$  is compact if $\X$ is compact. Alexandrov spaces of cohomogeneity one have been studied in \cite{GS}.


\subsection{Proof of Corollary~\ref{C:EQUIV_ALEX_METRIC}}
Let $\M$ be a closed topological manifold with a cohomogeneity one action of a compact Lie group $\G$. If the action is equivalent to a smooth action, then it is well known that one can construct a $\G$-invariant Riemannian metric on $\M$. Since, $\M$ is compact, this Riemannian metric has a lower sectional curvature bound and hence $\M$ is an Alexandrov space. Suppose now that the $\G$ action is not equivalent to a smooth action. In this case, $\M$ has a group diagram satisfying the hypotheses of \cite[Proposition~5]{GS} and, by this result, $\M$ admits an invariant Alexandrov metric.
 \hfill$\square$


\section{Tools and further definitions} 
\label{S:TOOLS}
In this section we review some standard results for cohomogeneity one smooth manifolds in the context of topological manifolds. We will use these tools in the proof of Theorem~\ref{T:CLASSIF}.  

We first point out that all the propositions and lemmas used by Hoelscher in \cite{Ho2} to determine both the groups $\G$ that may act by cohomogeneity one on a smooth closed manifold $\M$ and the fundamental group of $\M$ also hold for topological manifolds. Indeed, the fact that $\M$ is a union of  two mapping cylinders is a key point in the proofs of  most statements in \cite{Ho2}. By Mostert's work \cite{Mo}, this is also the case for a cohomogeneity one topological manifold. We collect the relevant results here for easy reference, focusing our attention on the cases where at least one of $\K^\pm/\HH$ is the Poincar\'e sphere.

The following proposition determines when two different group diagrams yield the same manifold. Its proof follows as in  \cite[Theorem~IV.8.2]{Br}, after  observing that a cohomogeneity one topological manifold decomposes as the union of two mapping cylinders.

\begin{prop} If a cohomogeneity  one topological manifold is given by a group diagram $(\G,\HH,\K^{-},K^{+})$, then any of the following operations on the group diagram will result in a $\G$-equivariantly homeomorphic topological manifold:
\begin{enumerate}
 	\item Switching $\K^{-}$ and $\K^{+}$,
	\item Conjugating each group in the diagram by the same element of $\G$,
	\item Replacing $K^{-}$ with $gK^{-}g^{-1}$  for $g\in N(\HH)_0$.
\end{enumerate}
Conversely, the group diagrams for two $\G$-equivariantly homeomorphic cohomogeneity one, closed topological manifolds must be mapped to each other by some combination of these three operations.
\end{prop}

The following result of Parker \cite[Proposition~1.8]{Ho2} is stated for smooth cohomogeneity one manifolds but the proof carries over to the topological category.


\begin{prop}[Parker]
\label{P:EXCEP_ORBIT}
A closed simply connected cohomogeneity one topological manifold has no exceptional orbits.
\end{prop}

The van Kampen Theorem applied to a closed cohomogeneity one manifold written as a union of two mapping cylinders yields the following result (cf.~\cite[Proposition~1.8]{Ho2}).


\begin{prop}[van Kampen Theorem]
\label{P:van_Kampen}
Let $\M$ be the closed cohomogeneity one topological  manifold given by the group diagram $(\G, \HH,  \K^-, \K^+)$ with   $\dim (\K^\pm/\HH) \geq 1$. 
Then $\pi_1(\M)\cong \pi_1(\G/\HH)/\N^-\N^+$, where
\begin{equation*}
\N^\pm= \ker\{\pi_1(\G/\HH)\to \pi_1(\G/\K^\pm)\}= \mathrm{Im} \{\pi_1(\K^\pm/\HH)\to \pi_1(\G/\HH)\}.
\end{equation*}
In particular $\M$ is simply connected if and only if the images of $\K^\pm/\HH$
generate $\pi_1(\G/\HH)$ under the natural inclusions.
\end{prop}

As the homogeneous spaces $\K^\pm/\HH$ are homeomorphic to either spheres or the Poincar\'e homology sphere, their fundamental groups are $\Z$, the identity or the binary icosahedral group. Since these groups are finitely generated, the next lemma follows as in the proof of \cite[Lemma~1.10]{Ho2}.


\begin{lem}
\label{L:SIMP_CONN_COND}
Let $\M$ be the cohomogeneity one topological  manifold given by the group diagram $(\G, \HH,  \K^-, \K^+)$ with at least one of $\K^\pm/\HH$ homeomorphic to $\PS^3$. Denote $\HH_\pm= \HH\cap \K^\pm_0$, and let $\alpha^i_\pm: [0,1]\to \K^\pm_0$ be curves that generate $\pi_1(\K^\pm/\HH)$, with $\alpha^i_\pm(0)=1\in \G$. The manifold $\M$ is simply connected if and only if 
\begin{enumerate}
\item $\HH$ is generated as a subgroup by $\HH_-$ and $\HH_+$, and
\item $\alpha^{i}_-$ and $\alpha^{i}_+$ generate $\pi_1(\G/\HH_0)$.
\end{enumerate}
\end{lem}


Recall that a cohomogeneity one action on a closed manifold $\M$ is \emph{non-primitive} if for some diagram $(\G,\HH,\K^-,\K^+)$ for $\M$ the isotropy groups $\K^\pm$ and $\HH$ are contained in some proper subgroup $\LL$ of $\G$. Such a non-primitive action is well known to be equivalent to the usual $\G$ action on $\G \times_\LL \M_\LL$, where $\M_\LL$ is the cohomogeneity one manifold given by the group diagram $(\LL,\HH,\K^-, \K^+)$. 


A cohomogeneity one action of $\G$ on a closed topological manifold $\M$ is \emph{reducible} if there is a proper normal subgroup of $\G$ that still acts by cohomogeneity one with the same orbits. Conversely, there is a natural way of extending an arbitrary cohomogeneity one action to an action by a possibly larger group. Such extensions, called \emph{normal extensions}, are described as follows (cf.~\cite[Propositions 11--13]{GS} and \cite[Section 1.11]{Ho2}). Let $\M$ be a cohomogeneity one topological manifold with group diagram  $(\G_1, \HH_1,  \K^-_1, \K^+_1)$ and let $\LL$ be a compact connected subgroup of $N(\HH_1) \cap  N(\K_1^-) \cap N(\K^+_1)$. Notice that $\LL \cap \HH_1$ is normal in $\LL$ and let $\G_2= \LL/(\LL \cap \HH_1)$. We then define an action by $\G_1\times \G_2$ on $\M$ orbitwise by 
 \[
 (\hat{ g_1}, [l])\cdot g_1(\G_1)_x =\hat{ g_1}g_1l^{-1}(\G_1)_x
 \]
 on each orbit $\G_1/(\G_1)_x$ for $(\G_1)_x= \HH_1$ or $\K^{\pm}_1$.  
 
 Notice that every reducible action is a normal extension of its restricted action. Therefore it is natural to consider non-reducible actions in the classification. We will use the following result on reducible actions (cf.~\cite[Proposition~11]{GS} and \cite[Proposition~1.12]{Ho2}) in the proof of Theorem~\ref{T:CLASSIF}.


\begin{prop}\label{reducing action} 
Let $\M$ be a cohomogeneity one manifold given by the group diagram $(\G, \HH,  \K^-, \K^+)$ and suppose that $\G_1\times \G_2$ with $\Proj_2(\HH)=\G_2$. Then the subaction of $\G_1\times 1$ on $\M$  is also by cohomogeneity one, with the same orbits, and with isotropy groups  $\K^{\pm}_1= \K^\pm \cap (\G_1\times 1)$ and $\HH_1= \HH \cap (\G_1\times 1)$.
\end{prop}

 The following  two propositions  give restrictions on the groups that may act by cohomogeneity one on a closed topological manifold. The next proposition can be found in  \cite{Br} for locally smooth actions. Here we prove it in the slightly more general case of topological actions on topological manifolds. 
 
 
\begin{prop}
\label{L:DIM_BOUND}
If  a compact connected Lie group $\G$ acts (almost) effectively on a topological $n$-manifold with principal orbits of dimension $k$, then $k\leq \dim \G \leq k(k+1)/2$.
\end{prop}


 \begin{proof}
Let $\G/\HH$ be a principal orbit. Since  $\dim \G/\HH=k$, the left inequality is immediate. To verify the right inequality, it suffices to know that $\G$ acts almost effectively on principal orbits, since then we can equip $\G/\HH$ with a $\G$-invariant Riemannian metric and obtain a homomorphism  $\varphi:\G\to\Isom(\G/\HH)$
with finite kernel $\K$. It then follows that $\G/\K\cong \varphi(\G)\leq \Isom(\G/\HH)$. Since $\K$ is finite, 
\begin{align*}
\dim \G = \dim \G/\K & \leq \dim \Isom(\G/\HH)\\
				    & \leq \frac{k(k+1)}{2},
\end{align*}
where the last inequality follows from a well-known theorem of Myers and Steenrod \cite{MS}. To finish the proof, let us show that $\G$ acts almost effectively on principal orbits. As mentioned in Section~\ref{S:Preliminaries},  all principal isotropy groups are conjugate to each other and conjugate to a subgroup of the  singular isotropy groups. As a result, $\G$ acts almost effectively on the principal orbits. 
\end{proof}


An argument as in the proof of \cite[Proposition 1.18]{Ho2} yields the following lemma:
\begin{lem}\label{P:ABELIAN_G}
Let $\M$ be a closed, simply connected topological manifold with an (almost) effective cohomogeneity one action of a compact Lie group $\G$. Suppose that the following conditions hold:
\begin{itemize}
	\item $\G=\G_1\times \T^m$ and $\G_1$ is semisimple;
	\item $\G$ acts non-reducibly;
	\item at least one of the homogeneous spaces $\K^\pm/\HH$ is the Poincar\'e sphere.
\end{itemize} 
Then, $\G_1\neq 1$ and $m\leq 1$. Moreover, if $m=1$, then   one of the homogeneous spaces $\K^\pm/\HH$ is a circle.
\end{lem}

It is well known that every compact connected Lie group has a finite cover of the form $\G_{ss} \times \T^ k$ , where $\G_{ss}$ is semisimple and simply
connected and $\T^ k$  is a torus. The classification of compact simply connected
semisimple Lie groups is also well known. We refer the reader to \cite[Section 1.24]{Ho2} for a list of such groups and their subgroups needed in our classification. We also have the following proposition (cf.~\cite[Proposition 1.25]{Ho2}), which gives further restriction on the groups. 


\begin{prop}
 Let $\M$ be the cohomogeneity one topological manifold given by the group diagram $(\G, \HH,  \K^-, \K^+)$, where $\G$ acts non-reducibly on $\M$. 
 Suppose that $\G$ is the product of groups
\begin{equation*}
\label{P:BOUND_DIMH}
	\G=\Pi_1^{i}(\SU(4))\times \Pi_1^{j}(\G_2)\times \Pi_1^{k}(\Sp(2))\times \Pi_1^{l}(\SU(3))\times \Pi_1^{m}(\RS^3)\times (\RS^1)^n.
\end{equation*}
 Then
\begin{equation*}
	\dim \HH \leq 10i+8j+6k+4l+m.
\end{equation*}
\end{prop}


We conclude this section with an observation on groups acting on the Poincar\'e homology sphere.


\begin{lem} 
\label{L:P3_TRANS}
Let $\G$ be a compact Lie group of dimension at most $2$. If $\SU(2)\times \G$ acts transitively on $\PS^3$, then $\G$ acts trivially on $\PS^3$ and the isotropy group of the $\SU(2)\times \G$ action is  $\I^*\times \G$.
\end{lem}


\begin{proof}
Assume first that $\G$ is connected and let $\K$ be the kernel of the action. By Theorem~\ref{T:BREDON_P3}, $(\SU(2)\times \G)/\K$ is isomorphic to $\SU(2)$. Hence, $\dim \G= \dim \K$. Since $\K$ is a normal and connected subgroup of $\SU(2)\times \G$, $\Proj_1(\K)$ is a normal connected subgroup of $\SU(2)$. Thus $\Proj_1(\K)$ is trivial, as $\dim \G\leq 2$. As a result, $\K= 1\times \G $ and $\HH=\I^*\times \G$, where $\HH$ is the principal isotropy group.

Suppose now that  $\G$ is not connected. In this case,  $\SU(1)\times (\G)_0$ is connected and acts transitively on $\PS^3$ as a restriction of the action of $\SU(2)\times \G$. Therefore, $\I^*\times (\G)_0\subseteq \HH\subseteq \I^*\times \G$. Connectedness of the quotient $\PS^3=(\SU(2)\times \G)/\HH$ gives $\HH=\I^*\times \G$.
\end{proof}


\section{Proof of Theorem~\ref{T:CLASSIF}}
\label{S:CLASSIF}

Let $\G$ be a compact, connected Lie group acting almost effectively, non-reducibly  and with cohomogeneity one on a closed, simply connected  topological $n$-manifold $\M^n$, $5\leq n\leq 7$. We assume that the action is non-smoothable. Hence, by Theorem~\ref{T:SMOOTH_CONDITION}, at least one of $\K^\pm/\HH$, say $\K^+/\HH$, is homeomorphic to $\PS^3$, the Poincar\'e homology sphere $\PS^3$.  We analyze each dimension separately.
\\


\noindent \textbf{Dimension $\mathbf{5}$.} By Proposition~\ref{L:DIM_BOUND}, we have $4\leq \dim \G\leq 10$. Hence, by \cite[1.24]{Ho2}, $\G$ is one of $(\RS^3)^m\times \T^n$,  $\SU(3)\times \T^n$ or $\Spin(5)$. From Proposition \ref{P:ABELIAN_G}, we see that $n\leq 1$. Since $\dim \HH=\dim \G-4$, Proposition  \ref{P:BOUND_DIMH} gives the possible groups. These are, up to a finite cover:  $\RS^3\times \RS^1$, $\RS^3\times \RS^3$, $\SU(3)$ and $\Spin(5)$. On the other hand, since $\K^+/\HH=\PS^3$, $\dim \K^+=3+\dim \HH=\dim \G-1$. Therefore,   $\G=\RS^3\times \RS^1$ is the only possibility, since the other groups do not have a subgroup of the given dimension.

Now we determine the group diagrams for  $\RS^3\times \RS^1$. From Proposition~\ref{P:ABELIAN_G} we have $\K^{-}/\HH=\RS^1$,  so $\HH_0=\{(1,1)\}$, $\K^{+}_0=\RS^3\times 1$ and $\K^{-}_0=\{(e^{ip\theta}, e^{iq\theta})\}$.  Thus, $\K^{+}=\RS^3\times \mathbb{Z}_k$ and $\HH=\I^*\times \mathbb{Z}_k$ by Lemma~\ref{L:P3_TRANS}. Since $\I^*\times \mathbb{Z}_k\subseteq \K^-\subseteq N_{\RS^3\times \RS^1}(\K^{-}_0)$,  the group $\K^{-}_0$ must be $1\times \RS^1$, which yields the following diagram:
\begin{align*}
(\RS^3\times \RS^1,  \I^*\times \mathbb{Z}_k,   \I^*\times \RS^1, \RS^3\times \mathbb{Z}_k). 
\end{align*}
Hence, $\M$ is equivalent  to $\PS^3\ast \SP^1$, which is homeomorphic to the double suspension of $\PS^3$. By the Double Suspension Theorem, $\M$ is  homeomorphic to $\SP^5$ and the action is the one described in Example~\ref{E:C1P3}.\\


\noindent \textbf{Dimension $\mathbf{6}$.} Proceeding as in the $5$-dimensional case, we find that $5\leq \dim \G\leq 15$ and $\dim \HH=\dim \G-5$. It follows from Propositions~ \ref{P:BOUND_DIMH} and \ref{P:ABELIAN_G} that $\G$ must be one of $\RS^3\times \RS^3$, $\RS^3\times \RS^3\times \RS^1$, $\SU(3)$, $\SU(3)\times \RS^1$, $\Sp(2)$, $\Sp(2)\times \RS^1$ or $\Spin(6)$. On the other hand, since $\K^+/\HH=\PS^3$, we must have that $\dim \K^+=\dim \G-2$. This dimension restriction rules out all possible groups except $\RS^3\times \RS^3$. 

Now we determine the possible diagrams for $\G=\RS^3\times \RS^3$.   First, suppose that $\K^{+}/\HH=\PS^3$ and $\K^{-}/\HH=\SP^l$, $l\geq 1$, so 
$\K^{+}_0=\RS^3\times \RS^1$ and $\HH_0=\{(e^{ip\theta}, e^{iq\theta})\}$. If $l\geq 2$, $\G/\K^{-}$ is simply connected, and consequently $\K^{+}$ is connected by Proposition \ref{P:van_Kampen}. Therefore,  $\K^{+}=\RS^3\times \RS^1$, which acts transitively on $\PS^3$. By Lemma \ref{L:P3_TRANS}, $\RS^1$ acts trivially on $\PS^3$ and  $\HH=\I^*\times \RS^1$. For $l=2$, $\K^{-}$ has to be $\I^*\times \RS^3$, and we have the following diagram:
\begin{align*}
(\RS^3\times \RS^3, \I^*\times \RS^1,   \I^*\times \RS^3, \RS^3\times \RS^1). 
\end{align*}
Therefore, $\M$ is equivalent to $\PS^3\ast \SP^2$ with the action described in Example \ref{E:C1P3}. 

For $l=3$ there is no subgroup of $\RS^3\times \RS^3$ containing $\HH=\I^*\times \RS^1$ such that $\K^{-}/\HH=\SP^3$. Accordingly, no finite extension of the maximal torus of $\RS^3\times \RS^3$ contains $\HH$, so $\K^{-}/\HH\neq \RS^1$. 

The only remaining case to be  considered is when $\K^{-}/\HH$ is also the Poincar\'e sphere, i.e.~$\K^{-}/\HH=\PS^3$. We now show that the only possible diagram that can occur is
\begin{align*}
(\RS^3\times \RS^3,  \I^*\times \RS^1,  \RS^3\times \RS^1, \RS^3\times \RS^1). 
\end{align*}

 First, notice that  since $\K^+_0=\RS^3\times \RS^1$ acts transitively on $\PS^3$, Proposition \ref{L:DIM_BOUND} implies that $\I^*\times \RS^1\subseteq \HH$. Since $\K^-_0$ is a $4$-dimensional subgroup of $\G$ containing $\HH$, it has to be $\RS^3\times \RS^1$ as well. On the other hand, since $\M$ is simply connected, $\HH$ must be generated by $\HH_+$ and $\HH_-$ as in Lemma~\ref{L:SIMP_CONN_COND}, so $\I^*\times \RS^1\subseteq \HH\subseteq \RS^3\times \RS^1$.  The fact that $\I^*$ is a maximal subgroup of $\RS^3$ implies that $\HH=\I^*\times \RS^1$ and, as a result, $\K^-=\K^+=\RS^3\times \RS^1$. On the other hand, the following action on $\Susp\PS^3\times  \SP^2$ gives rise to the same diagrams as the $\RS^3\times \RS^3$-action on $\M$:
\begin{align*}
	(\RS^3\times \RS^3)\times (\Susp \PS^3\times  \SP^2)	&\to \Susp \PS^3\times  \SP^2\\
	((g,h), ([x,t], y))											&\mapsto ([gx,t], hyh^{-1}).
\end{align*}
Therefore, $\M$ is homeomorphic to $\Susp\PS^3\times  \SP^2$, which is in turn homeomorphic to $\SP^4\times  \SP^2$ by Theorem \ref{T:TOP_MFT}. \\


\noindent \textbf{Dimension $\mathbf{7}$.} By Proposition \ref{L:DIM_BOUND} we know that $6\leq \dim \G\leq 21$ and $\dim \HH=\dim \G-6$. As before, Propositions  \ref{P:BOUND_DIMH} and \ref{P:ABELIAN_G} give us the possible acting groups:  $\RS^3\times \RS^3$, $\RS^3\times \RS^3\times \RS^1$, $\SU(3)$, $\RS^3\times \RS^3\times \RS^3$, $\SU(3)\times \RS^1$, $\Sp(2)$, $\SU(3)\times \RS^3$, $\Sp(2)\times \RS^3$, $\G_2$, $\SU(4)$, $\SU(4)\times \RS^1$ and $\Spin(7)$. Since $\dim \K^+=\dim \G-3$, we easily rule out most of the groups and the only possible groups remaining are $\RS^3\times \RS^3$, $\RS^3\times \RS^3\times \RS^1$ and $\RS^3\times \RS^3\times \RS^3$. We analyze each case separately.
\\


\noindent $\mathbf{G=S^3\times S^3.}$   Assume first that $\K^{+}/\HH=\PS^3$ and $\K^{-}/\HH=\SP^l$, $l\geq 2$. By Proposition \ref{P:van_Kampen}, $\K^+$ is connected. Thus, we have   $\K^{+}=\RS^3\times 1, \Delta \RS^3$,  and $\HH=\I^*\times 1,  \Delta \I^*$, respectively.   On the other hand,  since $\dim \HH=0$, and $\K^{-}/\HH$ is simply connected, $\K^-_0$ is also simply connected. A glance at the subgroups of $\RS^3\times \RS^3$  shows that only $\RS^3\times \RS^3$ and its $3$-dimensional subgroups  are simply connected. Since $(\RS^3\times \RS^3)/\HH$ is not a sphere, $\K^-$ is necessarily $3$-dimensional. Therefore $\K^-_0$ is one of $\RS^3\times 1$, $1\times \RS^3$ or $\Delta \RS^3$.  It is apparent that $(\RS^3\times 1)/\HH$ is not a sphere, so we are left with the two other cases. We easily rule out the the case $\K^-_0=\Delta \RS^3$, since otherwise  $\K^-$, being a subgroup of $N(\Delta \RS^3)=\pm \Delta \RS^3$,  has at most two components, while  $\pi_0(\K^-)=\I^*$. Hence, $\K^-=\I^*\times \RS^3$, and  we have the two following diagrams:
\begin{align*}
	(\RS^3\times \RS^3, \I^*\times 1, \I^*\times \RS^3, \RS^3\times1)
\end{align*}
and
\begin{align*}
(\RS^3\times \RS^3,  \Delta \I^*,  \I^*\times \RS^3, \Delta \RS^3),
\end{align*}
with the following actions, respectively:
\begin{align*}
(\RS^3\times \RS^3)\times ( \PS^3\ast  \SP^3)&\to \PS^3\ast  \RS^3\\
((g,h), ([x,y,t], y))&\mapsto [gx, hy,t]
\end{align*}
and
\begin{align*}
(\RS^3\times \RS^3)\times ( \PS^3\ast  \SP^3)&\to \PS^3\ast  \SP^3\\
((g,h), ([x,y,t], y))&\mapsto [gx, gyh^{-1},t].
\end{align*}

Now let $\K^{-}/\HH=\SP^1$, so $\K^{-}_0=\{(e^{ip\theta}, e^{iq\theta})\}$. Since 
\[
	\I^*\times 1\subseteq \K^{-}\subseteq N_{\RS^3\times \RS^3}(\K^{-}_0),
\] 
$\K^{-}_0$ must be $1\times \RS^1$. Hence we have the diagram
\begin{align*} 
(\RS^3\times \RS^3, \I^* \times  \mathbb{Z}_k,  \I^* \times \RS^1, \RS^3\times \mathbb{Z}_k).
\end{align*}
In this case, the action is non-primitive. In fact, for $\LL= \RS^3\times \RS^1\subseteq \RS^3\times \RS^3$, we have the following diagram, which is the diagram of the   cohomogeneity one action of $\RS^3\times \RS^1$ on $\PS^3\ast \SP^1$ already described in dimension $5$:
\begin{align*} 
	(\RS^3\times \RS^1,  \I^* \times  \mathbb{Z}_k,  \I^* \times \RS^1, \RS^3\times \mathbb{Z}_k).
\end{align*}
Therefore, $\M$ is an $\SP^5$-bundle over $\SP^2$.

Finally, suppose that $\K^{-}/\HH=\PS^3$. If $\K^{+}_0= \RS^3\times 1$, then   there exists a subgroup of $\HH$, say  $\tilde{\HH}$, such that $\tilde{\HH}=\I^*\times 1$ and  $\K^{+}_0/\tilde{\HH}=\PS^3$.  Now we have two possibilities: either $\I^*\times 1\subseteq \K^{-}_0$ or $\I^*\times 1\nsubseteq \K^{-}_0$. Assume first that  $\I^*\times 1\subseteq \K^{-}_0$.  Thus   $\K^{-}_0=\RS^3\times 1$. On the other hand, by Proposition~\ref{L:SIMP_CONN_COND}, $\HH$ is generated by $\HH\cap  (\RS^3\times 1)$, which gives that $\HH=\tilde{\HH}=\I^*\times 1$.  Therefore, we have the  following group diagram:
 \begin{align}
 \label{D:6D} 
 	(\RS^3\times \RS^3, \I^*\times 1, \RS^3\times 1,   \RS^3\times 1).
 \end{align}
On the other hand, $\Susp  \PS^3 \times \SP^3$ admits a cohomogeneity one action given by
\begin{align*}
	(\RS^3\times \RS^3)\times (\Susp  \PS^3 \times \SP^3)& \to \Susp  \PS^3 \times \SP^3\\
	((g,h),([x,t], y))& \mapsto ([gx,t],hy),
\end{align*}
which gives diagram~\eqref{D:6D} above. Thus $\M$ is equivariantly homeomorphic to  $\Susp  \PS^3 \times \SP^3$.

Now assume that $\I^*\times 1\nsubseteq \K^{-}_0$. Hence $\K^{-}_0$ has to be $1 \times \RS^3$ and consequently  $1\times \I^*\subseteq \HH$.  Therefore, the following diagram appears:
 \begin{align*}
	  (\RS^3\times \RS^3,      \I^* \times  \I^*, \I^* \times \RS^3,  \RS^3\times \I^*).
 \end{align*}
As a result, $\M$ is equivalent to $\PS^3\ast \PS^3$ with the action given by
\begin{align*}
(\RS^3\times \RS^3)\times (\PS^3\ast \PS^3)& \to \PS^3\ast \PS^3\\
((g,h),([x,y,t]))& \mapsto [gx,hy,t].
\end{align*}

Now let $\K^+_0=\Delta \RS^3$, so that $\Delta \I^*\subseteq \HH\subset \K^{-}$. Notice that by the classification of transitive and almost effective  actions on $\PS^3$, the group $\Delta \RS^3$ acts on $\PS^3$ in the natural way. As a result, $\K^{-}_0=\Delta \RS^3$. On the other hand, by Proposition \ref{L:SIMP_CONN_COND}, $\HH$ is generated by $\Delta \RS^3\cap \HH$, which is a finite diagonal subgroup of $\RS^3\times \RS^3$ containing $\Delta \I^*$. Let $\Delta \Gamma = \Delta \RS^3\cap \HH$. Therefore, $\I^*\subseteq \Gamma$. Then $\Gamma$ must be $\I^*$, for $\I^*$ is a maximal subgroup of $\RS^3$. Hence $\HH=\Delta \I^*$ and  we  have the following diagram:
\begin{align*}
	(\RS^3\times \RS^3, \Delta  \I^*,  \Delta \RS^3, \Delta \RS^3).
\end{align*}

Therefore, $\M$ is  equivariantly homeomorphic to $\Susp  \PS^3 \times  \SP^3$ with the action
\begin{align*}
	(\RS^3\times \RS^3)\times (\Susp  \PS^3 \times \SP^3)& \to \Susp  \PS^3 \times \SP^3\\
((g,h),([x,t], y))& \mapsto ([gx,t],gyh^{-1}).
\end{align*}
We conclude that $\M$ is homotopy equivalent to $\SP^4\times \SP^3$. \\


\noindent $\mathbf{G=S^3\times S^3\times S^1.}$  In this case $\dim\HH=1$ and $\K^{-}/\HH=\RS^1$ by Proposition \ref{P:ABELIAN_G}. Therefore, $\K^{-}_0$ would be a $2$-torus subgroup of  $\G$, and $\K^{+}_0\subseteq \RS^3\times \RS^3\times 1$. We can assume that $\K^{+}_0=\RS^3 \times \RS^1\times 1$, and consequently, $\I^* \times \RS^1\times 1\subseteq \HH$. Since $\I^* \times \RS^1\times 1\subseteq \K^{-}\subseteq N(\K^{-}_0)=N_\G(\T^2)$, $\K^{-}_0$ has to be $1\times \T^2$. Therefore, we have the following diagram:
\begin{align*}
	(\RS^3\times \RS^3\times \RS^1,  \I^*\times \RS^1\times \mathbb{Z}_k,  \I^*\times \T^2, \RS^3\times \RS^1\times \mathbb{Z}_k). 
\end{align*}
This action is a non-primitive action. Indeed, for $\LL=\RS^3\times \T^2\subseteq \RS^3\times \RS^3\times \RS^1$, we have diagram
\begin{equation*}
(\RS^3\times \T^2, \I^*\times \RS^1\times \mathbb{Z}_k, \I^*\times \T^2, \RS^3\times \RS^1\times \mathbb{Z}_k),
\end{equation*}
 which is a non-effective extension of the cohomogeneity one almost effective action of $\RS^3\times \RS^1$   on $\PS^3\ast \SP^1$ described in dimension $5$. Thus $\M$ is a $\PS^3\ast \SP^1$-fiber bundle on $(\RS^3\times \RS^3\times \RS^1)/(\RS^3\times \T^2)$. Therefore, $\M$ is homeomorphic to an $\SP^5$-bundle over $\SP^2$. 
\\


\noindent $\mathbf{G=S^3\times S^3\times S^3.}$  We show that no non-reducible diagram for this case occurs. In fact, we will show that all possible diagrams in this case  reduce to the diagrams of the case $\G=\RS^3\times \RS^3$. First, note that  $\dim \HH=3$ and $\dim \K^+=6$. Recall that $\Proj_l$, $l=1,2,3$, denotes  projection onto the $l$-th factor of $\RS^3\times \RS^3\times \RS^3$. Since we assume that the action is non-reducible, it follows from Proposition~\ref{reducing action} that $\Proj_l(\HH_0)$, $l=1, 2, 3$, is not $\RS^3$. On the other hand, $\Proj_l(\HH_0)$ cannot be trivial. Otherwise, $\HH$ would be a $3$-dimensional subgroup of $\RS^3\times \RS^3$ and $\HH_0$ must project onto one of the factors. This yields a reducible action, which contradicts the assumption that the action is non-reducible. Thus, $\Proj_l(\HH_0)=\RS^1$, for $l=1, 2, 3$. An inspection of the subgroups of $\RS^3\times \RS^3\times \RS^3$ shows that none of the  $6$-dimensional subgroups of $\RS^3\times \RS^3\times \RS^3$ contains $\HH$. Therefore, no non-reducible action can occur. 
\hfill$\square$


\section{Proof of Theorem~\ref{T:TOP_MFT}}
\label{S:TOP_MFT}

To prove that $\Susp \PS^3\times \SP^2$ is homeomorphic to $\SP^4\times\SP^2$ we will use a special instance of the classification of closed, oriented, simply connected $6$-dimensional topological manifolds with torsion free homology. This classification follows from work of  Wall \cite{Wa}, Jupp \cite{J}, and Zhubr \cite{Zh}. We first recall the following theorem of Jupp (cf.~ \cite[Theorem~1]{J}).


\begin{thm}[Jupp]
\label{T:Class. of 6- MFD}
Orientation-preserving homeomorphism classes of closed, oriented, $1$-connected $6$-manifolds $\M$ with torsion free homology correspond bijectively with isomorphism
classes of systems of invariants:
\begin{itemize}
	\item $r=\rank H^3(\M,\Z)$, a nonnegative integer;\\
	\item $H=H^2(\M, \Z) $, a finitely generated free abelian group;\\
	\item $\mu: H\oplus H\oplus H \to  \Z$, a symmetric trilinear form given by the cup product evaluated on the orientation class;\\
	\item $p_1(\M)\in H^4(\M,\Z)$, the first Pontryagin class;\\
	\item $w_2(\M)\in H^2(\M,\Z_2)$, the second Stiefel--Whitney class;\\
	\item $\Delta(\M)\in H^4(\M,\Z_2)$, the Kirby--Siebenmann class.
\end{itemize}
The systems of invariants must satisfy the equation
\begin{align*}
\mu(2x + W, 2x + W, 2x + W) \equiv (p + 24T) (2x + W) \pmod{48}
\end{align*}
for all $x\in H$, where $W\in H$, $T\in \Hom_{\Z}(H, \Z)$ reduce $\mathrm{mod}~2$ to $w_2, \Delta$. Such a manifold has
a smooth (or PL) structure if and only if $\Delta(\M)=0$, and the smooth structure is unique.
\end{thm}

We point out that for closed, oriented, simply connected $6$-dimensional topological manifolds the first Pontryagin class is always integral (see~\cite{J}). Okonek and Van de Ven  \cite[pp. 302--303]{OV} have summarized the classification in the special case where $r=0$ and $H=\Z$. This is the case that is relevant to us, since for $\Susp \PS^3\times \SP^2$ we have $r=0$ and $H=\Z$. We now recall these results.


\begin{prop}[\protect{\cite{OV}}]
\label{P:Invariants of Homotopy classes}
 Let $\M$ be as in Theorem~\ref{T:Class. of 6- MFD} with $r=0$ and $H=\Z$. The system of invariants introduced in Theorem \ref{T:Class. of 6- MFD} can be identified with $4$-tuples $(\bar{W},  \bar{T}, d, p)\in \Z_2 \times \Z_2 \times \Z \times \Z$, where the degree $d$ corresponds to the cubic form $\mu$. Such a $4$-tuple is admissible if and only if 
\begin{equation}\label{Relation}
d(2x + W)^3 \equiv (p + 24T) (2x + W)  \pmod{48},
\end{equation}
for every integer $x$.
\end{prop}


\begin{defn}[\protect{\cite{OV}}]
 Two admissible $4$-tuples $(\bar{W},  \bar{T}, d, p)$ and $(\bar{W'},  \bar{T'}, d', p')$ are equivalent if and only if $\bar{W}=\bar{W'}$, $\bar{T}=\bar{T'}$ and $(d', p')=\pm (d, p)$. 
\end{defn}


\begin{prop}[\protect{\cite{OV}}]
\label{Homeo Type: O-V}
The assignment 
\[
X\to (\bar{W},  \bar{T}, d, p).
\]
 induces a  $1$-$1$ correspondence between oriented homeomorphism classes of  closed,  oriented, simply connected $6$-dimensional topological manifolds with torsion free homology and equivalence classes of admissible systems of invariants, where $(\bar{W},  \bar{T}, d, p)$ is a normalized $4$-tuple, i. e. $d\geq 0$, and $p\geq 0$ if $d=0$. 
\end{prop}


\begin{defn}[\protect{\cite{OV}}]
 Two  normalized $4$-tuples $(\bar{W},  \bar{T}, d, p)$, $(\bar{W'},  \bar{T'}, d', p')$  are weakly equivalent if and only if $d'=d$, $\bar{W'}=\bar{W}$, $p+24T\equiv p'+24T' \pmod{48}$ if $d\equiv 0  \pmod{2}$, $p\equiv p' \pmod{24}$ if $d\equiv 1  \pmod{2}$ .
\end{defn}


\begin{prop}[\protect{\cite{OV}}]
\label{Homtopy Type: O-V}
The assignment 
\[
X\to (\bar{W},  \bar{T}, d, p).
\]
induces a  $1$-$1$ correspondence between homotopy classes of simply connected, closed, oriented, $6$-dimensional topological manifolds with torsion free homology and weak equivalence classes of admissible systems of invariants. 
\end{prop}

Now we use the above results to prove that $\Susp \PS^3\times \SP^2$ is homeomorphic to $\SP^4\times \SP^2$. Let  $(\bar{W},  \bar{T}, d, p)$ be the admissible $4$- tuple of $\M=\Susp \PS^3\times \SP^2$ as in Proposition \ref{P:Invariants of Homotopy classes}. Since $(0, 0, 0, 0)$ is the admissible $4$-tuple of $\SP^4\times \SP^2$, and $\M$ is homotopy equivalent  to $\SP^4\times \SP^2$, Proposition~\ref{Homtopy Type: O-V} implies that $d=0$, $\bar{W}=0$, and $p+24T\equiv  0 \pmod{48}$. On the other hand, in this case,  equation~\eqref{Relation} is equivalent to $p\equiv 0 \pmod{24}$, i.e. $p=24k$, for $k=0, \pm 1, \pm 2, \ldots$ Hence, $T\equiv p/24 \pmod{2}$; therefore, it suffices to compute the first Pontryagin class of $\M$. If $p=0$, then $\M$ is homeomorphic to $\SP^4\times \SP^2$ by Proposition \ref{Homeo Type: O-V}; if $p\equiv 0 \pmod{48}$,  $\M$ is smoothable and, if $p\equiv 24 \pmod{48}$, $\M$ is non-smoothable by Theorem \ref{T:Class. of 6- MFD}. 

 In the remainder of this section, we compute the first Pontryagin class $p_1$ of $\Susp \PS^3\times \SP^2$ and show that $p_1=0$. Note that rational Pontryagin classes are defined for topological manifolds and, more generally, for rational homology manifolds. One can also define Hirzebruch $l$-classes so that $l_1=\frac{1}{3}p_1$  (see \cite{Hi,K,M,MiS,Th,Za0}).  
 The $l$-classes are multiplicative, i.e.~given any two rational homology manifolds $\X$ and $\Y$, one has 
\begin{align*}
l_i(\X\times \Y)= \sum_{p+q=i}l_p(\X)l_q(\Y).
\end{align*}
Thus in our case, since $\Susp \PS^3$ is a rational homology manifold, we have \begin{align*}
l_1(\Susp \PS^3\times \SP^2)&= l_1(\Susp \PS^3)+ l_1(\SP^2)\\
&= l_1(\Susp \PS^3).
\end{align*}
Therefore, to find $p_1(\Susp \PS^3\times \SP^2)$, it suffices  to find $l_1(\Susp \PS^3)$.

First observe that $\Susp \PS^3\cong \SP^4/\I^*$, where $\I^*$ acts in the obvious way on a round $\SP^4$ by orientation preserving  diffeomorphisms. This observation allows us to use a formula of Zagier \cite{Zag} to compute the Hirzebruch $L$-class  of the quotient space $\X/\G$ of a closed oriented smooth manifold $\X$ by the  orientation preserving diffeomorphism action of a finite group $\G$. Since $\Susp \PS^3$ is $4$-dimensional, the top dimensional component of $L$ is $l_1$ and, by \cite{Zag}, $l_1$ equals the signature. Thus, to compute $l_1(\Susp \PS^3)$, we need only compute $\Sign(\Susp \PS^3)$. We do this using results of Atiyah and Singer \cite{AS} and of Atiyah and Bott \cite{AB}, which we briefly outline in the following paragraphs.

Let $\X$ be a closed, oriented, smooth manifold and $\G$ a finite group acting on $\X$ by orientation-preserving diffeomorphisms. Let 
\begin{align*}
\pi: \X\to \X/\G  
\end{align*}
be the projection map of $\X$ onto the orbit space $\X/\G$. As mentioned above,  the top dimensional component of $L(X/\G)$  is $\Sign(\X/\G)$. By the Atiyah-Singer $\G$-equivariant signature theorem (see \cite[Section~6]{AS} or \cite{Hi2}), the signature of $\X/\G$ is given by 
\begin{align}
\label{E:SIGNATURE}
\Sign(\X/\G)=\frac{1}{|\G|}\sum_{g\in \G}\Sign(g, \X).
\end{align}
In our particular case, where  $\G=\I^*$ acts on $\X=\SP^4$  with only two isolated fixed points, $\Sign(g,\X)$ is given by a signature formula of Atiyah and Bott \cite[Theorem~6.27]{AB}, which we state as Theorem~\ref{T:AB_SIGNATURE_FORMULA} below. Before quoting the theorem, we recall some notation.

Let $f: \X \to \X$ be an isometry of a compact, oriented even-dimensional Riemannian manifold $\X$ and $p$ be a fixed point of $f$. Consider the differential 
\begin{align*}
df_p: T_p\X \to T_p\X. 
\end{align*}
Because $f$ is an isometry of $X$, $df_p$ will be an isometry of $T_p\X$. Hence, one  may decompose $T_p\X$  into a direct sum of orthogonal $2$-planes 
\begin{align*}
 T_p\X = E_1 \oplus E_2 \oplus ... \oplus E_n,
 \end{align*}
which are stable under $df_p$. Let $(e_k, e'_k)$ be an orthogonal basis of $E_k$. We may choose $(e_k, e'_k)$ so that 
\begin{align*}
v_p(e_1\wedge e'_1\wedge ... \wedge e_n\wedge  e'_n)= 1,
\end{align*}
where $v$ is the volume form of $\X$.
Relative to such a basis $df_P$ is then given by rotations by angles $\theta_k$ in $E_k$. 
That is, \begin{align*}
df_p e_k &= \cos\theta_k e_k + \sin \theta_k e'_k \\[.2cm]
df_p e'_k &=- \sin\theta_k e_k + \cos \theta_k e'_k 
\end{align*}
The resulting set of angles $\{\theta_k\}$ is called a \emph{coherent system} for $df_p$. 


\begin{thm}[Atiyah and Bott]
\label{T:AB_SIGNATURE_FORMULA}
 Let $f: \X^{2n} \to \X^{2n}$ be an isometry of the compact oriented 
even dimensional Riemannian manifold $\X$. Assume further that $f$ has only 
isolated fixed points $\{p\}$, and let ${\theta_k^p}$ be a system of coherent angles for $df_p$. 
Then the signature of $f$ is given by 
 \begin{align*}
\Sign (f, \X) = \sum_p i^{-n} \prod_k  \cot (\theta_k^p /2).
\end{align*}
\end{thm}

We now use Theorem~\ref{T:AB_SIGNATURE_FORMULA} to compute $\Sign(g,X)$ for each $g\in \I^*$ and recover $\Sign(\SP^4/\I^*)$ via equation~\eqref{E:SIGNATURE}.  For non-trivial $g\in G$, the fixed point set $\X^g$ has two elements, say $\{p,q\}$. Let  $\{\alpha, \beta\}$ be the coherent system for $p$. Then the coherent system for $q$ will be $\{-\alpha, \beta\}$, so $\Sign(g, \SP^4)=0$. On the other hand, $\Sign(e, \SP^4)=\Sign(\SP^4)=0$ and $\Sign(-e, \SP^4)= 0$ by Theorem~\ref{T:AB_SIGNATURE_FORMULA}. As a result $\Sign(\SP^4/\I^*)=0$. Hence the top component of $L(\SP^4/\I^*)$ is zero. Since the top component of $L(\SP^4/\I^*)$ is  $3p_1(\SP^4/\I^*)=3p_1(\SP^4/\I^*\times \SP^2)$, we conclude that the first Pontryagin class of $\Susp \PS^3\times\SP^2$ is zero. Therefore, $\Susp \PS^3\times\SP^2$ is homeomorphic to $\SP^4\times \SP^2$.
\hfill$\square$



\end{document}